\documentclass[12pt]{amsart}
\usepackage{amsmath, amssymb, amsthm, amsfonts, mathrsfs}
\usepackage{cite}
\usepackage{amscd}
\usepackage{url}

\newtheorem{prop}{Proposition}

\newtheorem{theorem}[prop]{Theorem}

\theoremstyle{definition}

\newcommand{\R}{\mathbb{R}}
\newcommand{\ess}[1]{\hat{#1}}

\begin{document}
\title{The Replicator Equation as an Inference Dynamic}
\author{Marc Harper}
\address{University of California Los Angeles}
\email{marcharper@ucla.edu} 
% \authorinfo{\texttt{c}}
\date{\today}

\begin{abstract}
The replicator equation is interpreted as a continuous inference equation and a formal similarity between the discrete replicator equation and Bayesian inference is described. Further connections between inference and the replicator equation are given including a discussion of information divergences, evolutionary stability, and exponential families as solutions for the replicator dynamic, using Fisher information and information geometry. 
\end{abstract}

\maketitle

\section{Introduction}

To address the question of in what sense is natural selection related to information theory and statistical inference, we draw an analogy between Bayesian inference and models of natural selection available in evolutionary game theory. Exploring this requires the use of information theory which leads into the use of information geometry and the common geometric structure of information geometry and evolutionary game theory. Recognition of this framework leads to generalizations of Bayesian inference and is explored in future work.

Bayesian inference and the discrete replicator equation share a remarkable formal similarity. The continuous replicator dynamic can be analyzed using some of the same techniques used in Bayesian inference, such as by the Kullback-Liebler divergence and with exponential families. This paper describes these similarities and explains how information geometry illuminates the connection.

As a model of natural selection, the replicator dynamic models the informatic behavior of the population distribution. Emerging purely from information-theoretical constructions, the replicator requires no biological assumptions. This is because the geometry of evolutionary game theory comes from the information geometry of manifolds of probability distributions. The theoretical results associated to this geometry, such as Fisher's fundamental theorem, are facts describing the evolution of information captured by replicating systems.

\subsection{Bayesian Inference}

\begin{quote}
Inductive inference is the only process known to us by which essentially new knowledge comes into the world.
-- R. A. Fisher, \emph{The Design of Experiments} (1935) \cite{Fisher35}
\end{quote}

Bayesian inference is a discrete dynamical system utilizing Bayes' Theorem for iterative dynamic inference. It is widely used in machine learning, e.g. in spam filtering and document classification. Define the process as follows. Consider a collection of events $H_1, H_2, \ldots , H_n$, along with Bayes' theorem
\[ P(H_i | E) = \frac{P(E|H_i) P(H_i)}{P(E)} \qquad \text{for $i = 1, 2, \ldots , n$ where} \]
\begin{enumerate}
    \item the events constitute the entire state space: $\sum_{i=1}^{n}{P(H_i)} = 1$,
    \item $P(H_i)$ is the prior probability of $H_i$,
    \item $E$ is an event corresponding to newly encountered evidence and $P(E)$ is the marginal probability of $E$, where $P(E) = \sum_{i=1}^{n}{P(E | H_i) P(H_i)}$,
    \item $P(H_i|E)$ is the posterior probability of $H_i$ given the evidence $E$.
\end{enumerate}

The process adjusts the probabilities of the events $H_1, H_2, \ldots , H_n$ in light of the evidence provided by the observation $E$, forming a dynamic process \[(P(H_1), \ldots, P(H_n)) \to (P(H_1|E), \ldots, P(H_n|E)),\] which can be iterated over a sequence of observations $E_1, E_2, \ldots$. The Kullback-Liebler information divergence $D_{KL} \left( P(H|E) || P(H) \right)$ is used to measure the gain in information from passing to the posterior distribution.

\subsection{The Discrete Replicator Dynamic}

\begin{quote}
The theory of evolution by cumulative natural selection is the only theory we know of that is in principle capable of explaining the existence of organized complexity. -- Richard Dawkins, \emph{The Blind Watchmaker} (1987) \cite{Dawkins87}
\end{quote}

Consider a population of $n$ types of replicating entities, such as genotypes (e.g. possible viral genetic sequences) or phenotypes (e.g. eye color or investment strategies). Let $x_i$ be the proportion of the population of the $i$th type and denote the population distribution $x=(x_1, \ldots, x_n)$. The discrete replicator dynamic\cite{Cressman03} is:
% \begin{equation}\label{discrete_replicator_dynamic}
\[x_{i}' = \frac{x_i f_i(x)}{\bar{f}(x)}, \qquad \text{for $i = 1, 2, \ldots , n$ where} \]
% \end{equation}
\begin{enumerate}
 \item The types completely describe the population so that $\sum_{i}{x_i} = 1$, (i.e. the set of all possible states is the simplex),
 \item $f_i(x)$ is the fitness of type $i$ (dependent on the population distribution), $f=(f_1, \ldots, f_n)$ is the fitness landscape, 
 \item $\bar{f}(x) = \sum_{i=1}^{n}{x_i f_i(x)}$ is the average fitness, and
 \item $x_i'$ is the frequency of type $i$ in the next generation of the population, adjusted by the proportionality of fitness relative to the average population fitness.
\end{enumerate}

A population of a particular distribution obtains information about the environment from the fitness landscape. Replication adjusts the distribution of types in the population as dictated by the relative fitness as measured by the fitness landscape. In light of fitness landscape, some types proliferate while others decline, much like the probability of a particular event is adjusted by Bayes' theorem in light of new evidence in Bayesian inference.

\section{Formal Similarity of the Discrete Replicator Dynamic and Bayesian Inference}
The following dictionary describes the formal analogy of Bayesian inference and the discrete replicator dynamic. This analogy was independently discovered in \cite{Shalizi09}.

\begin{center}
\begin{tabular}{ll}
\textbf{Bayesian Inference} & \textbf{Discrete Replicator}\\ \hline
Prior Distribution $(P(H_1), \ldots, P(H_n))$ & Population state $x = (x_1, \ldots, x_n)$\\
New Evidence $P(E | H_i)$ & Fitness landscape $f_i(x)$\\
Normalization $P(E)$ & Mean fitness $\bar{f}(x)$\\
Posterior distribution $P(H_1 | E), \ldots, P(H_n | E)$ & Population state $x' = (x_1', \ldots, x_n')$
\end{tabular}
\end{center}

The fitness landscape provides the observation (evidence) in the inference process and the population re-aligns proportionally. Bayesian inference is a special case, formally, of the discrete replicator dynamic, since the fitness landscape in each coordinate may depend on the entire population distribution rather than only on the proportion of the $i$-type, which is significant if $n>2$. In the above formulation of Bayesian inference there is no explicit dependence of the probability $P(H_i | E)$ on any event $H_j$ with $i \neq j$.

% \section{Interpreting the Continuous Replicator Dynamic as an Inference Process}

Although the discrete replicator dynamic is intuitively satisifying, the nonlinearity of the discrete dynamic makes analysis difficult. A continuous version of the replicator equation is widely used and as a differential equation has more tractable tools for analysis.

\subsection{The Continuous Replicator Dynamic}

The replicator dynamic has an intuitive motivation as a general model of natural selection. Suppose a function $f$ from the simplex to $\R^n$ describes the fitness of the population states (or mixed-strategies), with the $i$-th component function $f_i$ giving the fitness of the $i$-th type, depending on the entire population distribution. The intuitive idea of natural selection, translated from Dawkins' Universal Darwinism, is that the relative rate of change of the proportion of the $i$-th type should be given by the difference of the fitness of the $i$-th type and the mean population fitness. In equations,
\[\frac{\dot{x}_i}{x_i} = f_i(x) - \sum_{i=1}^{n}{x_i f_i(x)} = f_i(x) - \bar{f}(x),\]
where $\bar{f}(x)$ denotes the mean of $f(x)$. After rearrangement, the replicator equation takes the form
% \begin{equation} \label{replicator_equation}
\[ \dot{x}_i = x_i(f_i(x) - \bar{f}(x)). \]
% \end{equation}

The continuous analog can be obtained by a limiting process from the discrete dynamic\cite{Cressman03}, followed by a change in velocity that does not alter the trajectories\cite{Hofbauer98} and a possible gauge transformation. Indeed, defining a differential equation from the difference equation,
%  (\ref{discrete_replicator_dynamic}),
\[ \dot{x}_i = \lim_{h \to 0}{\frac{x_i(t+h) - x_i(t) }{h} } \approx {x_i}' - x_i = x_i\frac{f_i(x)}{\bar{f}(x)} - x_i  = x_i\frac{f_i(x) - \bar{f}(x)}{\bar{f}(x)},\]
which is equivalent to the following equation after a change in velocity because $\bar{f}(x)$ can be assumed to be strictly positive:
\[ \dot{x}_i = x_i\left(f_i(x) - \bar{f}(x)\right).\]

This description, in light of the relationship to inference, identifies the replicator dynamic as a continuous inference process. Since information divergence plays an important role in Bayesian inference, it is natural to study its uses for the replicator dynamic. First we must define evolutionary stability.

\section{Evolutionary Stability}
A central question in evolutionary game theory is evolutionary stability\cite{Weibull97, Nowak06}. An evolutionarily stable state (ESS) of the replicator dynamic is a population distribution that is robust to invasion by mutant types.
% \footnote{A better term would be \emph{selectively stable state} since the replicator equation has no drift or diversification.}
Formally, a distribution $\ess{x}$ on the simplex is called an \emph{evolutionarily stable state} of the replicator dynamic if $\ess{x} \cdot f(x) > x \cdot f(x)$ (in some neighborhood of $\ess{x}$). This means that $\ess{x}$ is a \emph{better reply} to all neighboring strategies, and hence robust under the action of selection to the invasion of nearby mutant strategies. Evolutionarily stable states are asymptotic rest points of the replicator dynamic and correspond to the concept of strong stability in dynamical systems\cite{Cressman03, Hofbauer98}.

\subsection{Kullback-Liebler Divergence is a Lyapunov function for the Replicator Dynamic}

The following theorem shows that the Kullback-Liebler information divergence forms a Lyapunov function for the replicator dynamic, given an evolutionarily stable state. In fact, evolutionary stability is characterized by this property. A version of this theorem was proved in \cite{Akin79} and in \cite{Akin90}. A similar result is proven in \cite{Hofbauer98}, with the Lyapunov function $V(x) = \prod_{i}{x_{i}^{ \hat{x}_i}}$.

\begin{theorem}\label{ess_Lyapunov}
The state $\ess{x}$ is an interior ESS for the replicator dynamic if and only if $D_{KL}(\ess{x} || x)$ is a local Lyapunov function.
\end{theorem}
\begin{proof}
Let $V(x) = D_{KL}(\ess{x} || x) = \sum_{i}{\ess{x}_i \log{\ess{x}_i}} - \sum_{i}{\ess{x}_i \log{x_i}}$.
Then we have that
\begin{align*} \dot{V}(x) &= -\sum_{i}{ \ess{x}_i \frac{\dot{x}_i}{x_i}} = -\sum_{i}{ \ess{x}_i (f_i(x) - \bar{f}(x)) }\\
&= -\sum_{i}{ \ess{x}_i f_i(x)} + \sum_{i}{ \ess{x}_i\bar{f}(x)} = -\sum_{i}{ \ess{x}_i f_i(x)} + \left(\sum_{i}{ \ess{x}_i}\right)\bar{f}(x)\\
&= -\sum_{i}{ \ess{x}_i f_i(x)} + \bar{f}(x) = -(\ess{x} \cdot f(x) - x \cdot f(x)) < 0.
\end{align*}
The last inequality holds if and only if $\ess{x}$ is an ESS. Finally, by Jensen's inequality, $D_{KL}$ is minimized when $x = \ess{x}$, so it is a local Lyapunov function.
\end{proof}

Viewing the state $\hat{x}$ as the final or equilibrium distribution of the dynamic (analogous to the ``true distribution'' in an inference context), we can interpret the quantity $D_{KL}(\ess{x} || x)$ as the \emph{potential information} of the dynamic system. As the system converges, the potential information is decreasing and is minimized because it is a Lyapunov function.

This result is the continuous analog to the use of the Kullback-Liebler divergence as a measure of information gain of Bayesian inference. Within the neighborhood of the ESS, the system is minimizing the information divergence between the current population distribution and that of the selectively stable configuration, that is it is minimizing the potential information in the system. The theorem shows that this is an informatic characterization of evolutionary stability.

\subsection{Potential Information and the Discrete Replicator Dynamic}

The potential information $D_{KL}(\ess{x} || x)$ plays an analogous role for the discrete replicator dynamic.

\begin{theorem}\label{ess_Lyapunov}
Suppose that the fitness landscape is strictly positive, that is $f_i(x) > 0$ for all $i$ and $x$.
If the population distribution unfolds according to the discrete replicator dynamic then $\ess{x}$ is an interior ESS if and only if the potential information is decreasing along iterations of the dynamic.
\end{theorem}

\begin{proof}
First note that the ESS condition can be equivalently stated as
\[\frac{\ess{x} \cdot f(x)}{x \cdot f(x)} > 1\]
for all $x$ in a neighborhood of $\ess{x}$ (using the assumption that the fitness landscape is strictly positive).

Consider the difference in potential information of two successive states $P = D_{KL}(\ess{x} || x') - D_{KL}(\ess{x} || x)$. Assume that x is in the ESS neighborhood of $\ess{x}$. Then,
\begin{align*}
P &= \sum_{i}{\ess{x}_i \log{\ess{x}_i}} - \sum_{i}{\ess{x}_i \log{x_i'}} - \left( \sum_{i}{\ess{x}_i \log{\ess{x}_i}} - \sum_{i}{\ess{x}_i \log{x_i}} \right)\\
&= \sum_{i}{\ess{x}_i \log{x_i}} - \sum_{i}{\ess{x}_i \log{x_i'}}\\
&= \sum_{i}{\ess{x}_i \log{x_i}} - \sum_{i}{\ess{x}_i \log{\left(x_i \frac{f_i(x)}{\bar{f}(x)}\right)}}\\
&= - \sum_{i}{\ess{x}_i \log{\left(\frac{f_i(x)}{\bar{f}(x)}\right)}}\\
&\leq - \log{\left(\sum_{i}{\ess{x}_i \frac{f_i(x)}{\bar{f}(x)}}\right)} = -\log{\left( \frac{\ess{x} \cdot f(x)}{x \cdot f(x)} \right)} < 0, \\
\end{align*}
where the log is moved outside the sum using Jensen's inequality and the logarithm in the last line is positive by the ESS condition.
\end{proof}

The proof shows that the state $\ess{x}$ is an ESS if and only if the potential information is decreasing along iterations of the dynamic, once again giving a characterization of evolutionary stability.

\subsection{Exponential Families -- Solutions of the Continuous Replicator Equation}

In Bayesian Inference, an exponential family produces a conjugate prior which is also an exponential family, possibly of the same type. The analogous property for the continuous inference equation, the replicator dynamic, is to be of the form of an exponential family at each point on the trajectory. Exponential families are maximal entropy distributions \cite{Naudts08}, a property that corresponds with the intuitive explanation of the action of natural selection from the introduction.
Define an \emph{exponential family} to be a collection of distributions of the form
\[ p(x; \theta) = \text{exp}\left(C(x) + \sum_{i}{ \theta_i F_i(x) - \psi(\theta)} \right),\]
for functions $F_i$, $C$, and $\psi$ and parameter vector $\theta$. These are maximal entropy distributions with respect to constraints of the form $E\left[F_i(x)\right] = \lambda_i$ and can be derived with Lagrange multipliers.

The solutions of the replicator equation can be realized as exponential families \cite{Karev09, Nihat05, Akin82}. Let $x_i = \exp (v_i - G)$ with $\dot{v_i} = f_i(x)$ and $G(x)$ a normalization constant to ensure that the distribution sums to one. From the fact that $\sum_{i}{x_i} = 1$, $0 = \sum_{i}{\dot{x_i}}$ and so
\begin{align*}
0 = \sum_{i}{\dot{x}_i} &= \sum_{i}{\exp (v_i(x) - G(x)) (\dot{v}_i(x) - \dot{G}(x))}\\
&= \sum_{i}{x_i (\dot{v}_i(x) - \dot{G}(x))} = \sum_{i}{(x_i f_i(x))} - \dot{G}(x)\\
&= \bar{f}(x) - \dot{G}(x)
\end{align*}
Hence $\dot{G} = \bar{f}(x)$. Now $x_i$ satisfies
\[ \dot{x_i} = \exp (v_i(x) - G(x)) (\dot{v_i}(x) - \dot{G}(x)) = x_i (f_i(x) - \bar{f}(x)), \]
which is the replicator equation. In the case of a log-linear fitness landscape, explicit solutions can be derived \cite{Nihat05}. In this case, the equation for the variable $v$ can be reduced to a linear differential equation, which can be solved with eigenvalue methods.

An intuitive discussion of the above result is worthwhile. Entropy is a measure of the extent to which the distribution has ``spread out'' over the landscape. Natural selection acts to fit the available niches in a fitness landscape, arriving at the maximal entropy distribution allowed by the constraints of the landscape. In the absence of variation within the fitness landscape, the replicator dynamic is stable. In fact, if $f_i(x) = c$ for all $i$ and all $x$ then any $x$ is stationary (and a Nash equilibrium), and the solution with maximal entropy distribution is the uniform distribution, which follows directly from the definition of the exponential family. In the case of a variable fitness landscape, natural selection re-aligns the population distribution to fill out the landscape if not at equilibrium, at each point taking the maximal entropy distribution available within the constraints of the values of the $f_i(x)$.

\section{Understanding the Connection}
The connections between inference and evolutionary game theory are not just formal coincidence. Information geometry explains the commonality.

\subsection{Information Geometry}

The set of categorical distributions on $n$ variables forms a Riemannian manifold\cite{Amari93} via the Fisher information metric
\[ g_{ij}(x) = \mathbb{E}\left[ \frac{\partial \log p}{\partial x^i} \frac{\partial \log p}{\partial x^j}  \right]. \]

The exponential map of this manifold gives the exponential families of the previous section. In information geometry, the replicator equation is known as the \emph{natural gradient}. The gradient flow of the Fisher information metric is the replicator equation. The Fisher information metric can be obtained from the Hessian of the Kullback-Liebler information divergence, localizing the asymmetric information divergence to the symmetric Fisher information. The metric is known in evolutionary game theory as the Shahshahani metric and the manifold is identified with its embedding into the reals as the $(n-1)$-dimensional simplex. For explicit details see \cite{Harper09_ig_egt}.

In this context, the intreptation of the replicator equation as a continuous inference equation is natural, and the properties of the dynamic with respect to the information divergence and formal solutions as exponential families is less surprising. The replicator equation can now be understood as modeling the informational dynamics of the population distribution, moving in the direction of maximal local increase of potential with respect to the Fisher information, and ultimately converging to a minimal potential information state if a stablizing state (ESS) exists in the interior of the state space.

The conceptual situation is similar to that of the Price equation, a statistical relationship that models evolutionary processes but itself relies on no biological assumptions. Indeed, the Price equation is equivalent to the replicator equation\cite{Page02}. Similarly, the Shahshahani geometry of evolutionary game theory has a purely mathematical origin from information theory. This means that Kimura's maximum principle and Fisher's fundamental theorem are statements of mathematical and statistical facts that happen to model evolutionary processes rather than facts about natural selection itself.

\section{Discussion}

The replicator dynamic is a continuous inference dynamic. It is a process guided by the geometry of Fisher information. As a model of natural selection, the replicator equation captures the informatic change associated to the population distribution. The formal analogy of Bayesian inference and the discrete replicator dynamic leads to the interpretation and use of information theoretic quantities in evolutionary game theory.

In particular, the Kullback-Liebler information divergence can be used to define a potential information for replicator dynamics, both discrete and continuous, given an evolutionarily stable state. this property is minimized by the action of the dynamic and characterizes evolutionary stability. The concept of exponential families gives formal solutions to the continuous replicator dynamic.

\bibliography{ref}
\bibliographystyle{plain}

\end{document}